\documentclass[reqno]{amsart}
\usepackage{amsmath}
\usepackage{amsthm}
\usepackage{bbm}
\usepackage{amsfonts}
\usepackage{euscript}
\usepackage{amssymb}
\usepackage{enumerate}
\usepackage{mathrsfs}
\usepackage{mathtools}
\usepackage[english]{babel}
\usepackage{microtype}
\usepackage[backend=biber,style=alphabetic]{biblatex}
\newcommand{\field}[1]{\mathbb{#1}}
\newcommand{\R}{\field{R}}

\newtheorem{theorem}{Theorem}
\newtheorem{cor}{Corollary}
\newtheorem{lemma}{Lemma}
\newtheorem{definition}{Definition}

\usepackage{import}
\usepackage{xifthen}
\usepackage{pdfpages}
\usepackage{transparent}

\newcommand{\norm}[1]{\left\lVert#1\right\rVert}

\title{Necessary and sufficient conditions for $a$-contraction}
\author{Cooper Faile}
\address{Department of Mathematics, The University of Texas at Austin, Austin, TX 78712.}
\email{jcfaile@utexas.edu}

\date{\today}
\thanks{2010 \textit{Mathematics Subject Classification}. 35B35, 35L45, 35L65, 35L67}
\thanks{\textit{Key words and phrases}. Stability, Conservation law, Relative entropy, Entropy solution, Shock wave, Contraction}
\thanks{\textbf{Acknowledgements}: The author would like to thank his graduate advisor, Alexis Vasseur, for providing helpful conversations and support throughout. This work was partially supported by NSF-DMS Grants 1840314, 2219434, and 2306852.}

\begin{document}
\begin{abstract}
    In this paper we investigate the theory of $a$-contraction with shifts with the intention of extending it to intermediate families. 
    The theory of $a$-contraction with shifts is used to prove orbital $L^2$ stability to shock solutions of conservation laws. 
    In this setting there are strong results for scalar laws and the extremal families of $n\times n$ systems of conservation laws. 
    The only known results showing contraction of interior families are for the contact family the full Euler system \cite{sveuler} and the case of rich systems \cite{serrevasseur16}.
    This investigation culminates in finding necessary and sufficient conditions for which small shocks of general systems are local attractors with respect to the $a$-contraction theory. 
\end{abstract}
\maketitle

\tableofcontents

\section{Introduction}
We consider a 1-d system of conservation laws
\begin{equation} u_t + (f(u))_x = 0 \quad t > 0,\ x\in \R \label{eq:cv}\end{equation}
where $u = (u_1,\dots, u_n):[0,\infty) \times \R \to \mathcal{V}$, $\mathcal{V} \subset \R^n$ is open and connected, and $f \in C^4(\mathcal{V}; \R^n)$. 
We assume our system admits at least one convex entropy-entropy flux pair $(\eta, q)$ satisfying 
\begin{equation} \eta' f' =  q' \label{eq:entropy-def} \end{equation}
 and that our solutions $u$ distributionally satisfy the entropy inequality
\begin{equation} \eta(u)_t + q(u)_x \leq 0. \label{eq:entropy-ineq}\end{equation}
In this paper we examine the $L^2$ stability of shock waves using the $a$-contraction with shift method pioneered by Kang, Leger, and Vasseur in \cite{kang16,leger}.
These works established the orbital stability of shocks for scalar laws and the extremal (the $1^\text{st}$ and $n^\text{th}$ family) shocks of systems against $L^2$ perturbations possessing a certain degree of regularity called the ``strong trace'' property. 
These results have been extended to establish stability BV solutions to scalar laws and small BV solutions to $2\times 2$ systems by Chen, Golding, Krupa, and Vasseur \cite{gkv,ckv,krupascalar} (and recently this methodology has been applied to isothermal gas dynamics giving stability without a small BV condition on the initial data \cite{cheng}.) 
In this current work we are particularly interested in extending these results to the case of discontinuities of interior families for general systems with a single entropy.
In this general setting we only have $a$-contraction in the case of the contact discontinuity of Euler \cite{sveuler}. 
In doing this, we restrict ourselves to small perturbations of small shocks, in contrast with the works mentioned above. 
In this general case the Bressan, Crasta, and Piccoli $L^1$ theory is the state of the art and shows $L^1$ stability of solutions (including shocks of interior families) against small BV perturbations \cite{bressanL1}.

To date, the primary results concerning necessary and sufficient conditions for $a$-contraction are contained in \cite{kang16}.
As in \cite{kang16} we consider solutions $u$ with strong traces, a regularity property weaker than $\text{BV}_{loc}$.
\begin{definition} \label{def:traces}
    Let $u \in L^\infty(\R^+ \times \R;\mathcal{V})$. $u$ verifies the strong trace property if for any Lipschitz curve $h:\R^+ \to \R$ there exists two bounded functions $u_+,u_-:\R^+ \to \mathcal{V}$ such that
    $$ \lim_{n\to 0}\int_0^T \quad\sup_{\mathclap{y \in (0,1/n)}} |u(t, h(t) + y) - u_+(t)| + \sup_{\mathclap{y \in (0,1/n)}} |u(t,h(t) - y) - u_-(t)|\,dt = 0$$
    for any $T > 0$. 
\end{definition} 
The authors then considered stability against large perturbations $u$ of a fixed shock $(u_L, u_R, \sigma_{LR})$ in the space 
$$ \mathcal{S}_{weak} = \{u \in L^\infty(\R^+ \times \R;\mathcal{V}): u\text{ has strong traces}\}.$$
To discuss these results we must first define the relative entropy and other quantities for $a$-contraction. 
First, we note that the cone of entropies for our system contains all positive multiples of our fixed entropy $\eta$ and all linear forms on phase space.
From this, it immediately follows that for any fixed $v \in \mathcal{V}$ the relative entropy,
$$ \eta(u|v) = \eta(u) - \eta(v) - \nabla \eta(v)(u-v), $$
is also in the cone of entropies.
This entropy has the corresponding relative entropy flux
$$ q(u;v) = q(u) - q(v) -\nabla \eta(v)(f(u) - f(v)).$$
These entropy-entropy flux pairs were used in the original relative entropy method of Dafermos \cite{dafermos} and Diperna \cite{diperna} to show weak-strong uniqueness of Lipschitz solutions against distributional solutions to~\eqref{eq:cv}.
The program was moved into weak solutions of scalar conservation laws by Leger by considering the relative stability of shocks up to a shift \cite{legersolo} .
Following \cite{kang16} the theory can be extended to systems by constructing a pseudo distance to a fixed entropic shock $(u_L, u_R, \sigma_{LR})$ using a fixed weight $a > 0$
\begin{equation} E(t) = \int_{-\infty}^{h(t)} a \eta(u|u_L) \,dx + \int_{h(t)}^\infty \eta(u|u_R)\,dx \label{eq:dist}\end{equation}
where $h(t)$ is some Lipschitz curve dependant on our perturbation $u$. 
For any solution $u\in \mathcal{S}_{weak}$ to (\ref{eq:cv}) we have for almost all $t$
\begin{equation} \frac{d}{dt} E(t) \leq \dot h(t)( a\eta(u|u_L) - \eta(u|u_R) ) - (aq(u;u_L) - q(u;u_R)) \label{eq:dist-derivative} \end{equation}
due to the entropy inequality~\eqref{eq:entropy-ineq}. 
Our goal is to show whether there exists an appropriate shift $h(t)$ depending on the perturbation $u$ such that $E(t)$ is non-increasing for all time.
Under this assumption we immediately have that the shocks are orbitally stable in $L^2$ by the lemma
\begin{lemma}[\cite{leger}] 
    For any compact $V \subset\joinrel\subset \mathcal{V}$ there exists $c,\,C > 0$ such that for all $(u,v) \in \mathcal{V}\times V$ 
    \begin{equation} c|u-v|^2 \leq \eta(u|v) \leq C|u-v|^2. \label{lemma-1}\end{equation}
\end{lemma} 

To simplify notation we define 
\begin{align}
    \tilde \eta(u) &= a\eta(u|u_L) - \eta(u|u_R), \label{eq:til-eta}\\
    \tilde q(u) &= aq(u;u_L) - q(u;u_R), \label{eq:til-q} \\
    \Pi &= \{u| \tilde \eta(u) < 0\}. \label{eq:Pi}
\end{align}
For $a$-contraction we define the continuous and discontinuous dissipations
\begin{align}
    D_{cont}(u) &= -\tilde q(u) + \lambda_i(u) \tilde \eta(u),\label{eq:Dcont}\\
    D_{RH}(u_\pm, \sigma_\pm) &= (q(u_+;u_R) - \sigma_\pm \eta(u_+|u_R)) - a(q(u_-;u_L) - \sigma_\pm \eta(u_-|u_L)) \label{eq:DRH}.
\end{align}
These dissipations are exactly the upper bound in equation~\eqref{eq:dist-derivative} when $h(t)$ is traveling as a generalized characteristic of the $i^\text{th}$ family. 
Kang and Vasseur \cite{kang16} give necessary and sufficient conditions for $a$-contraction for perturbations in the class $\mathcal{S}_{weak}$
which amount to a collection of requirements on the dissipations $D_{cont}$ and $D_{RH}$ depending on the weight $a > 0$. 
This result is reproduced in section \ref{section:prelim}.
\cite{kang16} was able to use their sufficient condition to show extremal families $a$-contract for a large class of systems.
In the case of rich systems Serre and Vasseur were able to show any Liu-Majda stable shock is a local attractor for $a$-contraction by constructing two distinct entropies $\eta_+$, $\eta_-$ depending on the shock, including for intermediate families \cite{serrevasseur16}. 
In this paper the phrase ``local attractor'' should be interpreted as there exists an $\epsilon > 0$ such that the pseudo distance~\eqref{eq:dist} is non-increasing for any perturbation belonging to the smaller class
\begin{equation}\label{eq:local-att}
\begin{aligned}
    \mathcal{S}_{weak}^{\epsilon}(u_L,u_R) = \{u \in L^\infty(\R^+ \times \R;\mathcal{V}): u\text{ has strong traces and there exists}& \\
    \text{a Lipshitz $h$ such that }|u_+(t) - u_R| + |u_-(t) - u_L| \leq \epsilon \text{ for a.e. }t\}& 
\end{aligned}
\end{equation}
where $u_+, u_- \in L^\infty(\R^+; \mathcal{V})$ are the traces along the path $h$ as in Definition \ref{def:traces}. 

We remark that $S_{weak}^{\epsilon}$ is nontrivial for sufficiently small shocks $S$ due to the existence theory of, for instance, \cite{lewicka02} for small BV perturbations of discontinuous solutions.
In this example, we simply let $h$ follow the generalized $i^\text{th}$ characteristic emanating from the large discontinuity in the initial data.
Unlike in the rich case, for general systems we can only assume the existence of one strictly convex entropy $\eta$ limiting our choices to just positive multiples of $\eta$ with affine corrections.
At the moment, only the interior contact family of full Euler is known to $a$-contract with the use of a single entropy by Serre and Vasseur \cite{sveuler}.
Furthermore, Kang and Vasseur showed the 2D isentropic MHD is not stable in this sense by showing for all $a > 0$ the existence of a state $u \in \partial \Pi $ such that $D_{cont}(u) > 0$ \cite{kang16}.
However, this gives no information on states near the shock states, leaving open whether they are local attractors with respect to $a$-contraction.

The primary results of this paper are as follows. 
If our entropy and flux satisfy a pointwise condition then small $i^\text{th}$ shocks are local attractors of $a$-contraction; 
that is, as long as $|u_L - u_R|$ is sufficiently small, we have $a$-contraction against any perturbation in $S^{\epsilon}_{weak}(u_L,u_R)$ for some $\epsilon > 0$ sufficiently small.
\begin{theorem}
    Let the $i^\text{th}$ family be genuinely nonlinear. 
    Suppose at $u_L$ there exists $C \in \R$ such that the matrix
    $$ -C\nabla^2 \eta(u_L)(f'(u_L) - \lambda_i(u_L)I) + r_i(u_L)^t\nabla^2 \eta(u_L)f''(u_L) $$
    is negative definite on the subspace $\text{span}(\{r_k(u_L)\}_{k\ne i}).$
    Then there exists $s_0 > 0$ such that for all $0 < s < s_0$ there exists $\epsilon > 0, C > 0$ such that the shock $(u_L, S^i_{u_L}(s), \sigma(s))$ satisfies
    $$ \frac{d}{dt} E(t) \leq -C |u_-(t)-u_L|^2$$
    for all perturbations $u \in S^{\epsilon}_{weak}(u_L, S^i_{u_L}(s))$ with weight $a(s) = 1 + Cs$ and shift $h$, $u_-$ are the same functions as in~\eqref{eq:local-att}. 
    \label{thm:suf}
  \end{theorem}
We can construct systems which satisfy this above sufficient condition in neighborhoods of a specific state $u$ for intermediate families.
As a remark, we note that in the extremal case we can always find a $C$ satisfying the hypotheses of Theorem \ref{thm:suf}, as the matrix $f'(U) - \lambda_i(U) I$ is positive (resp. negative) definite on the subspace $\text{span}(\{r_k(u_L)\}_{k\ne i})$ for the extremal family $i = 1$ (resp. $i = n$.)

For contraction to hold for all sufficiently small shocks of size $s$ and weights $a = a(s)$ we have the following necessary condition.
\begin{theorem}
    Suppose there exists $s_0 > 0, u_L \in \mathcal{V}$ such that for all $0 < s < s_0$ the shock $(u_L, S^i_{u_L}(s), \sigma(s))$ is a local attractor with weight $a(s)$,
      where $a\in C^1([0,s_0))$ and $a(0) = 1$. 
      Then the matrix
      $$ -a'(0)\nabla^2 \eta(u_L)(f'(u_L) - \lambda_i(u_L)I) + r_i(u_L)^t\nabla^2 \eta(u_L)f''(u_L) $$
      is negative semidefinite on the subspace $\text{span}(\{r_k(u_L)\}_{k\ne i}).$
    \label{thm:nec}
\end{theorem}
We note that for weights with $a(0) \ne 1$ or $\lim_{s\to 0} |a'(s)| = \infty$ we cannot have $a$-contraction for small intermediate shocks, as this would cause $\text{diam}(\partial \Pi) \to 0$ as $s \to 0$ leading to a violation of ($\mathcal{H}1$) in Theorem \ref{thm:kv}. See \cite{kang16} for further details. 
This consideration greatly influences the strategy of our proofs, where we borrow many strategies and results from \cite{gkv}.
To date, these are the only results on linear weights ($a(s) = 1+Cs$) for small shocks.
Theorem \ref{thm:nec}'s converse can be made slightly stronger, as we actually need only consider states near $(u_L,u_R)$. 
As a corollary, we are able to show small shocks do not $a$-contract even with arbitrarily small perturbations. 
\begin{cor}
    Fix a weight $a$ satisfying the assumptions of Theorem \ref{thm:nec}. If the matrix
    $$ -a'(0)\nabla^2 \eta(u_L)(f'(u_L) - \lambda_i(u_L)I) + r_i(u_L)^t\nabla^2 \eta(u_L)f''(u_L) $$
    is not negative semidefinite then there exists $s_0 > 0$, depending on $a$, $u_L$, and $f$, such that for all $0 < s < s_0$ and $\delta > 0$ the shock $(u_L, S^i_{u_L}(s), \sigma(s))$ has near by states $(u_-,u_+)$ such that 
    $$ |u_- - u_L| + |u_+ - u_R| < \delta \quad \text{and} \quad u_+ = S_{u_-}^i(t),\,\text{for some } t > 0 $$
    for which $D_{RH}(u_{\pm},\sigma_\pm) > 0$. 
    \label{cor:converse}
\end{cor}
From this corollary we can then construct perturbations of our fixed shock with initial data 
\begin{equation} \label{eq:pert-construction}
     u_0^\delta(x) = \begin{cases}
    u_L & x < -2  \\
    u_- & -1 < x < 0 \\
    u_+ & 0 < x < 1 \\
    u_R & x > 2
\end{cases} 
\end{equation}
with the function smooth on the regions $[-2, -1]$ and $[1,2]$. 
By construction we then have $\frac{d}{dt} E(t) = D_{RH}(u_{\pm}) > 0$ where equality follows from the perturbation possessing just one shock for $t$ sufficiently small. 
Hence the shock $(u_L, u_R,\sigma_{LR})$ is not a local attractor. 
As an application of this corollary we show the MHD system examined in \cite{kang16} has many states $U_L$ for which small shocks with left state $U_L$ are not local attractors for any weight. 
This is in contrast with the extremal case, where generic extremal shocks $a$-contract against the large class $\mathcal{S}_{weak}$ when the system satisfies a few assumptions. 

\section{Preliminaries} \label{section:prelim}
For convenience, we recreate the necessary and sufficient criteria for $a$-contraction provided by Kang and Vasseur \cite{kang16}. First we must define ``$a$-relative entropy stable''
\begin{definition}
    An entropic Rankine-Hugoniot discontinuity $(u_L,u_R,\sigma_{LR})$ is $a$-relative entropy stable with respect to a weight $a > 0$ if it satisfies
    \begin{description}
        \item[$(\mathcal{H}1)$] For any $u \in \partial \Pi$ we have $D_{cont}(u) \leq 0$, 
        \item[$(\mathcal{H}2)$] For any entropic shock $(u_-,u_+,\sigma_\pm)$ such that $\tilde \eta(u_-) < 0 < \tilde \eta(u_+)$ we have $D_{RH}(u_\pm, \sigma_\pm) \leq 0.$
    \end{description}

\end{definition}
By means of~\eqref{eq:dist-derivative} this relative entropy stability property is related to the dissipations defined in the introduction by the following theorem. 
\begin{theorem}[\cite{kang16}]
    Fix $a > 0$. If an entropic Rankine-Hugoniot discontinuity $(u_L,u_R,\sigma_{LR})$ is such that
    given any $u \in \mathcal{S}_{weak}$ there exists a Lipshitz $h:[0,\infty) \to \R$ such that
    $$E(t) = \int_{-\infty}^{h(t)} a \eta(u|u_L) \,dx + \int_{h(t)}^\infty \eta(u|u_R)\,dx $$
    is non-increasing then $(u_L,u_R,\sigma_{LR})$ is $a$-relative entropy stable.

    \label{thm:kv}
\end{theorem}
In particular, $(\mathcal{H}2)$ is necessary because we can construct perturbations such as~\eqref{eq:pert-construction} where the upper bound in equation~\eqref{eq:dist-derivative} is sharp;
\begin{equation} \frac{d}{dt}E(t) = D_{RH}(u_\pm, \sigma_\pm). \end{equation}

To bound the discontinuous dissipation $D_{RH}$ we will use the following lemma, which quantifies the entropy loss across a shock. 
\begin{lemma}[\cite{leger}]
    Suppose $\eta,\,q$ are an entropy flux pair and $u,v \in \mathcal{V}$. Then for $1 \leq i \leq n$, 
    \begin{equation}
        q(S_u^i(s);v) - \sigma_u^i(s)\eta(S_u^i(s)|v) = q(u;v) - \sigma_u^i(s) \eta(u|v) + \int_0^s \dot \sigma(t)\eta(u|S_u^i(t))\,dt. \label{eq:quant-diss}
    \end{equation}
\end{lemma}
Formula~\eqref{eq:quant-diss} can be used to deduce the following relation on the dissipations~\eqref{eq:Dcont}, ~\eqref{eq:DRH}
\begin{lemma}[\cite{gkv}]
    Suppose $(\eta,q)$ are an entropy-entropy flux pair. For any $s \geq 0$ we have 
    \begin{equation}D_{RH}(u, S_u^i(s), \sigma(s)) = D_{cont}(u) + \int_0^s \dot \sigma(t)(\tilde \eta(u) + \eta(u | S_u^i(t)))\,dt. \label{eq:DRH-Dcont} \end{equation}
\end{lemma}

For further lemmas we wish to have that $\tilde \eta$ is convex. 
We compute $\nabla^2 \tilde \eta(u) = (a(s) - 1) \nabla \eta(u)$ and we assume without loss of generality that there exists $s_0 > 0$ such that $a(s) - 1 \geq 0$ for all $s < s_0$. 
This is trivial when $a'(0) > 0$. 
If $a'(0) < 0$ we can perform the change of variables $x = -y$, which will yield the new system 
$$ u_t - (f(u))_y = 0 $$
and swaps $u_L,u_R$, giving a new weight $\tilde a(s) = a(s)^{-1}$ which satisfies $\tilde a'(0) > 0$. 
Finally, in the case that $a'(0) = 0$ we know that either $a(s) - 1 \geq 0$ or $a(s) - 1 \leq 0$ infinitely often as $s \to 0$. 
If we restrict ourselves to considering only $s$ in one of these two cases the limit computation in section \ref{section:comp} still follow, as long as $0$ is an accumulation point for such $s$.

Golding, Krupa, and Vasseur further observed that, under the assumption that $\eta(u|S_u^i(s))$ is increasing, the dissipation $D_{RH}(u, S_u^i(s), \sigma_\pm)$ is maximized by at most a single state $u^+(u) := S_{u_-}^i(s^*)$ satisfying the implicit equations
\begin{align}
    \eta(u | u^+(u)) &= - \tilde \eta(u) \label{eq:u-plus-dist}\\
    f(u^+(u)) - f(u) &= \sigma_\pm (u^+(u) - u). \label{eq:u-plus-on-S}
\end{align}
for states $u \in \Pi$. 
We refer to the shock $(u,u^+(u))$ as the ``maximal shock'' for the $i^\text{th}$ family relative to our fixed shock $(u_L,u_R)$. 
We then define the corresponding dissipation to this state 
\begin{equation}D_{max}(u) := D_{RH}(u,S_u^1(s^*), \sigma_u^1(s^*)) = D_{RH}(u,u^+(u), \sigma_{\pm}).\label{eq:Dmax}\end{equation}
Differentiating the equations~\eqref{eq:u-plus-dist} and~\eqref{eq:u-plus-on-S} gives us the following relations on $\nabla u^+$
\begin{align}
    [f'(u) - \sigma_{\pm} I] + (u^+ - u)\otimes \nabla \sigma_\pm&=[f'(u^+) - \sigma_\pm I] \nabla u^+, \label{eq:rh-der} \\
    [\nabla \eta( u_R) - \nabla \eta(u^+)] + a [\nabla \eta(u) - \nabla \eta(u_L)] &= (u - u^+)^t \nabla^2 \eta(u^+) \nabla u^+. \label{eq:eta-der}
\end{align}

We also will also use several other identities primarily found in \cite{gkv} which for convenience are reproduced here. 
\begin{lemma}[\cite{gkv}]
    We take the same assumptions on our weight as in Theorem \ref{thm:nec}.
    For shock size $s>0$ sufficiently small and states $u \in \Pi$ we have 
    \begin{align}
        l_i(u)f''(u):r_i(u)\otimes v =& (l_i(u)r_i(u))(\nabla \lambda(u)v) \label{eq:sj} \\
        |\nabla u^+(u)| \leq& C\label{eq:u-plus-lip}\\
        \lim_{s \to 0} l_j(u_L)\nabla u^+(u_L)r_k(u_L) =& (1-\delta_{ik})l_j(u_L)r_k(u_L) \label{lemma-3}\\
        \nabla D_{max}(u) =& [\nabla \eta(u^+(u)) - \nabla \eta(u_R) - a [\nabla \eta(u) - \nabla \eta(u_L)] ] \label{eq:nabla-Dmax} \\
        &\quad (f'(u) - \sigma_{\pm} I ) \nonumber\\
        \nabla^2 D_{max}(u) =& [\nabla^2 \eta(u^+(u)) \nabla u^+(u) - a \nabla^2 \eta(u) ](f'(u) - \sigma_{\pm} I ) \label{eq:nabla2-Dmax} \\
        &+ [\nabla \eta(u^+(u)) - \nabla \eta(u_R) - a [\nabla \eta(u) - \nabla \eta(u_L)] ]\nonumber\\
        &(f''(u) - I\otimes \nabla \sigma_{\pm}) \nonumber
    \end{align}
    where we understand the term $f'':v\otimes w$ in equation~\eqref{eq:sj} in the sense of the Fr\'{e}chet derivative. 

\end{lemma}

\begin{proof}
    We establish~\eqref{eq:sj} by differentiating $f'r_i = \lambda_i r_i$ in the direction $v$ yielding
    $$f'':r_i(u) \otimes v = r_i(u) (\nabla \lambda_i(u) v) + (\lambda_i(u)I - f'(u)) (\nabla r_i(u) v) $$
    then left multiplying by $l_i(u)$. 

    The remaining equations contained in the lemma are proven in section 6 of \cite{gkv} and follow identically in our setting.

\end{proof}

To compute the limit contained in section \ref{section:comp} we need the following result on the derivative of $\sigma_{\pm}$, the velocity of a maximal shock $(u,u^+(u), \sigma_\pm)$. 
\begin{lemma}
    \begin{equation} \lim_{s \to 0} \nabla \sigma^i(u,u^+(u))|_{u = u_L} = \frac{1}{2}\nabla \lambda_i(u_L) \left(I + \lim_{s \to 0} \nabla u^+(u_L)\right)\label{lemma-5} \end{equation}
\end{lemma}
\begin{proof}
    We recall the averaged matrices 
    $$A(u,v) = \int_0^1 f'(tv + (1-t)u)\,dt.$$
    These matrices have the property that any states $(u_-,u_+,\sigma)$ satisfying the Rankine-Hugoniot condition correspond to eigenvectors of $A$:
    $$ A(u_-,u_+)r = \sigma r\,\quad \text{where } r = \frac{u_+ - u_-}{|u_+ - u_-|},$$
    where the shock velocity $\sigma$ is an eigenvalue of $A$. 
    In the case that $u_- = u,u_+ = u^+(u)$ we find $\sigma = \sigma_i(u,u^+(u))$ and the above equation differentiates to give us 
    $$ r\otimes \nabla [\sigma_i(u,u^+(u))] = \nabla [A(u,u^+(u))] r + (A(u,u^+(u)) - \sigma_i(u,u^+(u)) I) \nabla r$$
    for states $u \in \Pi$.
    Left multiplying by $l$, the left eigenvector of $A(u,u^+(u))$ with the same eigenvalue $\sigma_i$ and with normalization $lr = 1$, the above becomes
    $$ \nabla \sigma_i(u,u^+(u)) = l \nabla A(u,u^+(u)) r.$$
    Since $f\in C^4(\mathcal{V})$ we can use the Leibniz rule to compute $\nabla A$ and we find
    \begin{align*}
        \nabla \sigma_i(u,u^+(u)) &= l(u,u^+(u)) \int_0^1 \nabla f'(tu + (1-t)u^+(u))\,dt\,r(u,u^+(u)) \\
        &= \int_0^1 l(u,u^+(u)) f''(tu + (1-t)u^+(u)) :r(u,u^+(u)) \\
        &\quad\quad\quad\quad\quad\quad\quad\quad\quad\quad\quad\quad\otimes [tI + (1-t) \nabla u^+(u)]\,dt \\
        &= \int_0^1 l_i(u) f''(u) :r_i(u) \otimes [tI + (1-t) \nabla u^+(u)]\,dt + O(|u-u^+(u)|)
    \end{align*}
    where we have used that $f''$ is Lipschitz and $|r(u,v) - r_i(u)| + |l(u,v) - l_i(u)| \leq C|u-v|$. 
    Applying equation~\eqref{eq:sj} the integral evaluates to give us
    \begin{align*}
        \nabla \sigma_i(u,u^+(u)) &= \frac{1}{2} \nabla \lambda_i(u)[I + \nabla u^+(u)] + O(|u-u^+(u)|).
    \end{align*}
    From here we can evaluate the above at $u = u_L$ and compute the limit as $s \to 0$ using the fact that $|u-u^+| \to 0$ as $s \to 0$ due to $\lim_{s\to 0}\tilde \eta(u) = 0$, and equations~\eqref{eq:u-plus-dist},~\eqref{lemma-1} to arrive at equation~\eqref{lemma-5}. 

   \end{proof}

Finally, we note that for our purposes we can remove the assumption that $\eta(u|S_u^i(t))$ is increasing for all $t$, as we consider just small perturbations of small shocks.
\begin{lemma} \label{lem:remove-assum-j}
    There exists $R, \overline s,s_0 > 0$ such that for all $u \in B_R(u_L) \cap \Pi$ and $s < s_0$ there exists a unique $s^*(u) < \overline s$. 
\end{lemma}
\begin{proof}
    We fix any $ R> 0 $ such that $B_{2R}(u_L) \subset\joinrel\subset \mathcal{V}$. 
    Computing the derivative of $s\mapsto \eta(u|S_u^1(s))$ we find 
    \begin{align*}
        \frac{d}{ds}\eta(u|S_u^i(s)) =& -\nabla^2 \eta(S_u^i(s)) : (u - S_u^i(s)) \otimes (\partial_t S_u^i(s)) \\
        =& s\nabla^2 \eta(u) : r_i(u)\otimes r_i(u) + O(s^2)
    \end{align*}
    and by continuity it follows there exists $\overline s > 0$ such that, for all $s \in (0,\overline s)$ and $u\in B_r(u_L)\cap \Pi$ we have $(d/ds) \eta(u|S_u^1(s)) > 0$.
    Since $s^*$ is defined by the implicit equation 
    $$ \eta(u|S_u^1(s^*)) = - \tilde \eta(u) $$
    this suffices to show uniqueness of any solution $s^* < \overline s$. 

    To show existence we note at $s = 0$ we have 
    $$ 0 = \eta(u|S_u^1(0)) < -\eta(u) $$
    for all $u \in B_r(u_L) \cap \Pi$. 
    At $s = \overline s$ we first show a bound on $-\tilde \eta(u)$. 
    Letting $\overline u \in \partial\Pi$ be the minimizer of $\overline u \mapsto |u-\overline u|$ we have the bound
    $$ -\tilde \eta(u) \leq -\tilde \eta(\overline u) -\nabla \tilde \eta(\overline u)(u - \overline u) = |\nabla \tilde \eta(\overline u)| d(u, \partial \Pi)$$
    by convexity of $\tilde \eta(u)$. 
    Furthermore, for sufficiently small $s_0$ we have $B_R(u_L) \cap \partial \Pi \ne \varnothing$, hence $|u-\overline u| < R$ and by the triangle inequality $\overline u \in B_{2R}(u_L).$
    From this we observe
    $$ |\nabla \tilde \eta(\overline u)| = |(a(s) - 1)(\nabla \eta(u) - \nabla \eta(u_L)) +(\nabla \eta(u_R) - \nabla \eta(u_L))| \leq K[2R(a(s) - 1) +s]$$
    where $K > 0$ is the Lipschitz constant of $\nabla \eta$ on $B_{2R}(u_L)$. 
    The right hand side converges to zero linearly as $s \to 0$, so we further restrict ourselves to $s_0$ sufficiently small to satisfy 
    $$ K[2R(a(s) - 1) +s] \leq Cs_0 < c \overline s^2 \leq \eta(u|S_u^1(\overline s))$$
    where $c$ is the constant from Lemma \ref{lemma-1} with $V = \overline{B_{2R}(u_L)}$.
    This suffices to show that 
    $$ - \tilde \eta(u) \leq \eta(u|S_u^1(\overline s)) $$
    and the existence of $s^*$ now follows by the intermediate value theorem.  
\end{proof}

\section{Computing the limit of $\nabla^2 D_{max}(u_L)/s$} \label{section:comp}
In this section we take $u_L \in \mathcal{V}$ to be fixed, $u_R = S_{u_L}^i(s)$ for any $s \in (0,s_0)$, and weight $a = a(s)$ such that $a \in C^1([0,s_0))$ and $a(0) = 1$. 
Throughout this section we assume $s_0$ is sufficiently small to satisfy Lemma \ref{lem:remove-assum-j} to guarantee existence and uniqueness of $u^+(u)$.  
From equation~\eqref{eq:nabla2-Dmax} the hessian of $D_{max}$ at $u_L$ is
\begin{equation} \nabla^2 D_{max}(u_L)= \left[(\nabla u^+(u_L))^t\nabla^2 \eta(u_R)  - a(s) \nabla^2\eta(u_L)\right](f'(u_L) - \sigma_{LR}I). \label{eq:Dmaxhess} \end{equation}
We note that in the case where the $i^\text{th}$ family is linearly degenerate we could equivalently take $u^+(u) = u$ (or any other value on $R_u^i$) as $D_{RH}(u,R_u^i(s))$ is constant in $s$ by equation~\eqref{eq:DRH-Dcont}. 
We elect to keep $u^+$ as defined in~\eqref{eq:u-plus-dist} to compute the hessian simultaneously for both linearly degenerate and genuinely nonlinear families. 
Rewriting (\ref{eq:rh-der}) with $u = u_L$ we have
\begin{equation}\label{eq:u+val}
\begin{aligned}
    \nabla u^+(u_L)  = I + [f'(u_L) - \sigma_{LR} I]&^{-1} [ (u_R - u_L)\otimes \nabla \sigma_\pm   \\
    &+  (f'(u_L) - f'(u_R)) \nabla u^+ ]
\end{aligned}
\end{equation}
where we note $f'(u_L) - \sigma_{LR} I$ is invertible when the $i^\text{th}$ family is genuinely nonlinear for sufficiently small $s$ due to shock satisfying the Lax condition, $$\lambda_i(u_R) < \sigma_{LR} < \lambda_i(u_L),$$ and our system being strictly hyperbolic. 
In the case of the $i^\text{th}$ family being linearly degenerate we restrict ourselves to computing the hessian on $\text{span}(\{r_j(u_L)\}_{j \ne i})$ where the operator is invertible. 
Substituting~\eqref{eq:u+val} into~\eqref{eq:Dmaxhess} we arrive at the formula
\begin{align*}
    \nabla^2 D_{max}(u_L) =& \left[(\nabla u^+)^t\nabla^2 \eta(u_L)  - a(s) \nabla^2\eta(u_L)\right](f'(u_L) - \sigma_{LR}I) \\
    &+ (\nabla u^+)^t [\nabla^2 \eta(u_R) - \nabla^2 \eta(u_L)](f'(u_L) - \sigma_{LR}I) \\
    =& \left[\nabla^2 \eta(u_L)  - a(s) \nabla^2\eta(u_L)\right](f'(u_L) - \sigma_{LR}I) \\
    &+ (\nabla u^+)^t(\nabla^2\eta(u_R) - \nabla^2\eta(u_L))(f'(u_L) - \sigma_{LR}I) \\
    &+  \left[ (u_R - u_L)\otimes \nabla \sigma_\pm +  [f'(u_L) - f'(u_R)] \nabla u^+ \right]^t\nabla^2 \eta(u_L) 
\end{align*}
where we used symmetry of $\nabla^2 \eta(u) f'(u)$ on the last line. 

Letting $j,k = 1,\dots, n$ we wish to take the limit as $s\to 0$ of the expressions 
\begin{align*}
    \frac{1}{s}r_j(u_L)^t&\nabla^2 D_{max}(u_L) r_k(u_L) \\
    =& \frac{1}{s}(\lambda_k(u_L) - \sigma_{LR})r_j(u_L)^t\left[1 - a(s) \right]\nabla^2 \eta(u_L) r_k(u_L) \\
    &+ \frac{1}{s}(\lambda_k(u_L) - \sigma_{LR})r_j(u_L)(\nabla u^+)^t(\nabla^2\eta(u_R) - \nabla^2\eta(u_L))r_k(u_L) \\
    &+  \frac{1}{s}r_k(u_L)^t \nabla^2 \eta(u_L) [ (u_R - u_L)\otimes \nabla \sigma_\pm +  [f'(u_L) - f'(u_R)] \nabla u^+ ]r_j(u_L) \\
    =:& D_{jk} + E_{jk} + F_{jk}
\end{align*}

\textbf{Case 1: $j,k \ne i$.}
For both $D_{jk}, E_{jk}$ we have one term that is $O(s)$ while the rest are $O(1)$ allowing us to compute our limit by simply differentiating
\begin{align*}
    \lim_{s\to 0} D_{jk} =& \lim_{s\to 0}\frac{1}{s}(\lambda_k(u_L) - \sigma_{LR})r_j(u_L)^t\left[1 - a(s) \right]\nabla^2 \eta(u_L) r_k(u_L) \\
    =& -(\lambda_k(u_L) - \lambda_i(u_L)) a'(0) r_j(u_L)^t \nabla^2\eta(u_L) r_k(u_L) \\
    =& - a'(0) \nabla^2\eta(u_L)(f'(u_L) - \lambda_i(u_L)I):r_j(u_L)\otimes r_k(u_L) \\
    \lim_{s\to 0} E_{jk} =& \lim_{s\to 0} \frac{1}{s}(\lambda_k(u_L) - \sigma_{LR})r_j(u_L)^t(\nabla u^+)^t(\nabla^2\eta(u_R) - \nabla^2\eta(u_L))r_k(u_L)\\
    =& -(\lambda_k(u_L) - \lambda_i(u_L)) \nabla^3 \eta(u_L): r_k(u_L)\otimes \lim_{s\to 0} \nabla u^+(u_L)r_j(u_L) \otimes r_i(u_L)\\
    =&-(\lambda_k(u_L) - \lambda_i(u_L)) \nabla^3 \eta(u_L): r_k(u_L)\otimes r_j(u_L) \otimes r_i(u_L)
\end{align*}
where the final equality is due to~\eqref{lemma-3}. 
To proceed we find an identity which combines $E_{jk}, F_{jk}$ into our desired quantity. 
Differentiating the identity $0 = \nabla^2 \eta(u):r_i(u) \otimes r_k(u)$ in the direction $r_j(u)$ we find 
\begin{align*}
    \nabla^3 \eta(u):r_i(u) \otimes r_j(u) \otimes r_k(u) =& -r_i(u)^t \nabla^2\eta(u)(\nabla r_k(u) r_j(u)) \\
    &-r_k(u)^t \nabla^2\eta(u)(\nabla r_i(u) r_j(u)).
\end{align*}
Substituting equation~\eqref{eq:sj} into the above expression and recalling that $l_i(u) \parallel r_i(u)^t \nabla^2 \eta(u)$ we find 
\begin{align*}
    \lim_{s\to 0} E_{jk}=& r_i(u_L)^t \nabla^2 \eta(u_L) f''(u_L):r_k(u_L)\otimes r_j(u_L) \\
    &- r_k(u_L)^t \nabla^2 \eta(u_L) f''(u_L):r_i(u_L)\otimes r_j(u_L).
\end{align*}
Finally, for $F_{jk}$ we find
\begin{align*}
    \lim_{s\to 0} F_{jk} =& \lim_{s\to 0} \frac{1}{s}r_k(u_L)^t \nabla^2 \eta(u_L) [ (u_R - u_L)\otimes \nabla \sigma_\pm +  [f'(u_L) - f'(u_R)] \nabla u^+ ]r_j(u_L) \\
    =& -[r_k(u_L)^t \nabla^2 \eta(u_L):r_i(u_L)] \nabla \sigma_{\pm} r_j(u_L) \\
    &+ r_k(u_L)^t \nabla^2 \eta(u_L) f''(u_L):r_i(u_L)\otimes r_j(u_L) \\
    =& r_k(u_L)^t \nabla^2 \eta(u_L) f''(u_L):r_i(u_L)\otimes r_j(u_L)
\end{align*}
where the first term is zero because $r_k(u_L)^t \parallel l_k(u_L)$. 
Hence we have 
$$ \frac{1}{s}r_j^t\nabla^2 D_{max} r_k \to \left[- a'(0) \nabla^2\eta (f' - \lambda_iI) + r_i^t \nabla^2 \eta f'' \right] :r_k\otimes r_j.$$
where all functions are evaluated at $u_L$. 

\textbf{Case 2: $j = i$.}
Due to symmetry of $\nabla^2 D_{max}(u_L)$ it suffices to just consider $j = i$. 
For $E$ and $D$ we find
$$ \lim_{s \to 0} D_{ik} = \lim_{s \to 0} E_{ik} = 0$$
due to $\sigma_{LR} \to \lambda_i(u_L)$ and $|\nabla u^+(u_L)| \lesssim 1$. 
For $F_{ik}$ we can directly compute
\begin{align*}
    \lim_{s\to 0} F_{ik} =& \lim_{s\to 0} \frac{1}{s}r_k(u_L)^t \nabla^2 \eta(u_L) \left[ (u_R - u_L)\otimes \nabla \sigma_\pm +  [f'(u_L) - f'(u_R)] \nabla u^+ \right]r_i(u_L) \\
    =& \lim_{s\to 0} r_k(u_L)^t \nabla^2 \eta(u_L) \frac{1}{s}(u_R - u_L) (\nabla \sigma_\pm r_i(u_L)) \\
    &+  r_k(u_L)^t \nabla^2 \eta(u_L) \frac{1}{s}[f'(u_L) - f'(u_R)] (\nabla u^+r_i(u_L))\\
    =& -r_k(u_L)^t \nabla^2 \eta(u_L) r_i(u_L) (\lim_{s\to 0} \nabla \sigma_\pm r_i(u_L)) \\
    &+  r_k(u_L)^t \nabla^2 \eta(u_L) f''(u_L): (\lim_{s\to 0} \nabla u^+r_i(u_L))\otimes r_i(u_L)\\
    =& \begin{cases} 0&k\ne i \\ -\frac{1}{2}\nabla \lambda_i(u_L)r_i(u_L)\left[r_i(u_L)^t \nabla^2 \eta(u_L) r_i(u_L)\right] & k=i\end{cases}
\end{align*}
due to equations~\eqref{eq:sj} and~\eqref{lemma-3}.

\section{Proof of main theorems}
\begin{proof}[Proof of Theorem \ref{thm:nec}]
    By Theorem \ref{thm:kv} the $a$-contraction property implies our shock $(u_L,u_R,\sigma_{LR})$ satisfies the $(\mathcal{H}2)$ property of $a$-relative entropy stable: $$D_{RH}(u_\pm,\sigma_\pm) \leq 0$$ for all entropic shocks $(u_-,u_+,\sigma_\pm)$ with $\tilde \eta(u_-) < 0 <  \tilde \eta(u_+)$. 
    From equation~\eqref{eq:u-plus-dist} $u^+(u_L) = u_R$ hence $D_{max}(u_L) = 0$.
    Furthermore, by continuity of $u^+$ and $\tilde \eta(u_L) < 0 < \tilde \eta(u_R)$ we find for $u$ sufficiently close to $u_L$ that $\tilde \eta(u) < 0 <  \tilde \eta(u^+(u))$ as well.
    Hence $D_{max} \leq 0$ on a neighborhood of $u_L$ and it follows that $D_{max}$ must have a local maximum at $u = u_L$ for all $0<s < s_0$.
    Therefore, $\nabla^2 D_{max}(u_L)$ is negative semidefinite for $s \in (0,s_0)$.
    Following the computation in section \ref{section:comp} reveals
    \begin{align*} 
        \frac{1}{s} \nabla^2 D_{max}(u_L):v\otimes v \to& [-a'(0)\nabla^2\eta(U_L)(f'(U_L)- \lambda_i(U_L)I)  \\
        &+ r_i(U_L)^t\nabla^2\eta(U_L)f''(U_L)]:v\otimes v 
    \end{align*}
    as $s \to 0$ for all $v \in V = \text{span}\{r_k(u_L)\}_{k\ne i}.$
    It follows that the limiting matrix is negative semidefinite on $V$.
\end{proof}

Before proving the sufficient condition we first establish that shocks $(u_-,u_+,\sigma_{\pm})$ sufficiently close to the fixed discontinuity $(u_L,u_R,\sigma_{LR})$ of the $i^\text{th}$ family are necessarily also of the $i^\text{th}$ family, as long as $|u_R- u_L|$ is sufficiently small. 
\begin{lemma} \label{lemma:sep}
    For any $u_L \in \mathcal{V}$, $i \in \{1,\dots,n\}$ there exists $s_0 > 0$, such that for all $u_R = S_{u_L}^i(s), 0<s<s_0$ there exists $\epsilon > 0$ such that any Rankine Hugoniot discontinuity $(u_-,u_+,\sigma_{\pm})$ with $u_-\in B_\epsilon(u_L),u_+\in B_\epsilon(u_R)$ must be of the of the $i^\text{th}$ family. 
\end{lemma}
\begin{proof}
    We choose $s_0$ sufficiently small to guarantee separation of characteristic speeds,
        $$ \delta :=\quad  \inf_{\mathclap{\substack{j\ne i,\\ u,v \in B_{2s_0}(u_L)}}} \quad|\lambda_i(u) - \lambda_j(v)| > 0.$$
    Fix $\epsilon > 0$. 
    Let $(u_-,u_+, \sigma_{\pm})$ be a Rankine Hugoniot discontinuity and suppose $u_-\in B_\epsilon(u_L),u_+\in B_\epsilon(u_R)$.
    We have the Rankine Hugoniot equations 
    \begin{align*}
        f(u_R) - f(u_L) &= \sigma_{LR}(u_R-u_L) \\
        f(u_+) - f(u_-) &= \sigma_{\pm}(u_+-u_-)
    \end{align*}
    and after combining we find, 
    \begin{align*}
        2\norm{f}_{C^{0,1}(B_{2s_0}(u_L))} \epsilon &\geq \norm{f}_{C^{0,1}(B_{2s_0}(u_L))}(|u_R - u_+| + |u_L - u_-|) \\
        &\geq |(f(u_R) - f(u_L)) - (f(u_+) - f(u_-))| \\
        &= |(\sigma_{LR} - \sigma_\pm + \sigma_\pm)(u_R-u_L) - \sigma_{\pm}(u_+-u_-)|\\
        &\geq ||\sigma_{LR} - \sigma_\pm| |u_R-u_L| - |\sigma_{\pm}| |(u_+-u_-) - (u_R - u_L)| |\\
        \shortintertext{Taking $\epsilon > 0$ sufficiently small and $(u_-,u_+,\sigma_\pm)$ to be of the $j\ne i^{\text{th}}$ family, this is bounded below by} 
        &\geq \frac{\delta}{2}|u_R-u_L|, 
    \end{align*}
    which is impossible for $\epsilon$ sufficiently small, forcing $(u_-,u_+,\sigma_\pm)$ to be of the $i^{\text{th}}$ family.
\end{proof}

Selecting $s$ and $\epsilon$ as in Lemma \ref{lemma:sep} we find $\dot h(t)$ is the Rankine-Hugoniot velocity of the $i^\text{th}$ shock $u_\pm$ for almost all time, when $u_-\in B_\epsilon(u_L)$, $u_+\in B_\epsilon(u_R)$. 
\begin{lemma}[Lemma 6 of \cite{leger}] \label{lemma:leger-follow-shock}
    Consider any Lipschitz path $h:[0,\infty) \to \R$ and $u$ an entropic weak solution to~\eqref{eq:cv} verifying the strong trace property. 
    Then for almost all $t > 0$ 
    \begin{align*}
        f(u(t,h(t)+)) - f(u(t,h(t) - )) &= \dot h(t) [u(t,h(t)+) - u(t,h(t) - )] \\
        q(u(t,h(t)+)) - q(u(t,h(t) - )) &\leq \dot h(t) [\eta(u(t,h(t)+)) - \eta(u(t,h(t) - ))] \\
    \end{align*}
\end{lemma}

Equipped with the above lemmas it suffices to consider just the discontinuities of the $i^\text{th}$ family occurring in our perturbation. 
\begin{proof}[Proof of Theorem \ref{thm:suf}]
    By equations~\eqref{eq:Dmax} and~\eqref{eq:nabla-Dmax} $D_{max}(u_L), \nabla D_{max}(u_L) = 0$ for all $s$. 
    Our assumption that $$-C\nabla^2 \eta(U_L)(f'(U_L) - \lambda_i(U_L)I) + r_i(U_L)^t\nabla^2 \eta(U_L)f''(U_L)$$
    is strictly negative definite guarantees there exists $s_0 > 0$ such that $\nabla^2 D_{max}(u_L)$ is strictly negative definite for all $s < s_0$, as
    $$ \lim_{s \to 0 }\frac{1}{s}\nabla D_{max}(u_L):v\otimes v  \text{ is negative definite} $$
    due to the computation in section \ref{section:comp}. 
    From this we can perform a Taylor expansion and find there exists an $\epsilon > 0$ sufficiently small such that
    \begin{align*}
        D_{max}(u) \leq& \nabla^2 D_{max}(u_L):(u-u_L)^{\otimes 2} + K|u-u_L|^3 \\
        \leq& - C|u-u_L|^2
    \end{align*} 
    for all $|u-u_L| < \epsilon$ and some $C > 0$.

    We now fix the neighborhoods of our shock states $U_{L} = B_{u_{L}}(\epsilon)$, $U_{R} = B_{u_{R}}(\epsilon)$. 
    We further take $\epsilon$ and $s_0$ to be sufficiently small to satisfy Lemma \ref{lemma:sep}. 
    Since $u \in \mathcal{S}_{weak}^{\epsilon}$ we have, for almost all $t > 0$, that $(u(t, h(t) -),u(t, h(t) +), \dot h(t))$ is an $i^\text{th}$ shock by Lemma \ref{lemma:sep} and Lemma \ref{lemma:leger-follow-shock}. 
    Hence by equation~\eqref{eq:dist-derivative} we find for almost all $t > 0$ that
    $$ \frac{d}{dt} E(t) \leq D_{RH}(u(t, h(t) -),u(t, h(t) +), \dot h(t)) \leq D_{max}(u(t, h(t) -)) \leq 0$$
    establishing the contraction. 
\end{proof}

\begin{proof}[Proof of corollary \ref{cor:converse}]
    Section \ref{section:comp} and our assumption that 
    $$ -a'(0) \nabla^2 \eta(u_L)(f'(u_L) - \lambda_i(u_L)I) + r_i(u_L)^t \nabla^2 \eta(u_L)f''(u_L) $$
    is not negative semidefinite on the subspace $V = \text{span}(\{r_j(u_L)\}_{j \ne i})$ guarantees the existence of an $s_0 > 0$ such that, for all $0 < s < s_0$, the matrix 
    $$ \nabla^2 D_{max}(u_L)\text{ is not negative semidefinite.} $$
    Hence for each $s < s_0$ there exists a unit vector $v \in V$ such that $\nabla^2 D_{max}(u_L):v\otimes v > 0$.
    By Taylor expanding in phase space around $u = u_L$ we find 
    \begin{align*}
        D_{max}(u) \leq& D_{max}(u_L) + \nabla D_{max}(u_L)(u - u_L) + \frac{1}{2} \nabla^2 D_{max}(u_L):(u - u_L)^{\otimes 2} \\
        &+ \norm{\nabla^3 D_{max}(u)}_{B_r(u_L)} |u-u_L|^3 \\
        =& \frac{1}{2} \nabla^2 D_{max}(u_L):(u - u_L)^{\otimes 2} + \norm{\nabla^3 D_{max}(u)}_{B_r(u_L)} |u-u_L|^3 \\
        \intertext{where again, $r$ is such that $B_r(u_L) \subset\subset\mathcal{V}$. Letting $u = u_L + tv$ for $t < r$ we find}
        =& \frac{1}{2} t^2\nabla^2 D_{max}(u_L):v^{\otimes 2} + t^3\norm{\nabla^3 D_{max}(u)}_{B_r(u_L)}|v|^3 \\
        \leq& \frac{1}{4} t^2\nabla^2 D_{max}(u_L):v^{\otimes 2}
    \end{align*}
    for $t$ sufficiently small. 
    Further, using the Lipshitz estimate in equation~\eqref{eq:u-plus-lip} and recalling that $u^+(u_L) = u_R$ we find
    $$ |u^+(u) - u_R| = |u^+(u) - u^+(u_L)| \leq C|u - u_L|. $$
    We therefore have 
    $$ |u^+(u) - u_R| + |u - u_L| \leq (C+1) |t| $$
    which can be made less than $\delta$ for $|t|$ sufficiently small.
\end{proof}

\section{A family of systems satisfying the hypotheses of Theorem \ref{thm:suf}}
As mentioned in the introduction, in the case of extremal shocks a constant $C\in \R$ satisfying the condition of Theorem \ref{thm:suf} trivially exists due to $f'(u) - \lambda_i(u)I$ being positive/negative definite for these families.
In the case of intermediate families we can construct systems satisfying the above sufficient condition at a point.
Consider the $3\times 3$ systems with state variables $U = (u, v, w) \in \R^3$,
$$ \begin{cases} 
    u_t + [(u+1)^2 + v(2\alpha w - 2u)]_x &= 0 \\
    v_t + [v^2 - u^2 -3w^2 +2\alpha uw]_x &= 0 \\
    w_t + [(w-1)^2 + v(2\alpha u - 6w)]_x &= 0
\end{cases} $$
where $\alpha \in \R$ is constant. 
Computing the jacobian of our flux $f$, we find
$$ f'(U) = \begin{pmatrix}
    2(u+1) - 2v & -2u + 2\alpha w & 2\alpha v \\
    -2u + 2\alpha w & 2v & -6w +2\alpha u \\
    2\alpha v & -6w +2\alpha u & 2(w-1) - 6v
\end{pmatrix} $$
which is symmetric, hence we can take $\eta(U) = \frac{1}{2} |U|^2$ as our entropy. 
Evaluating the matrix at $U_L = (0,0,0)$ we find our eigenvectors are $l_i(U_L)^t = r_i(U_L) = e_i$ with eigenvalues $\lambda_1(U_L) = 2$, $\lambda_2(U_L) = 0$, and $\lambda_3(U_L) = -2$, so additionally we are hyperbolic in a neighborhood of $U_L$.
We will let this neighborhood be our $\mathcal{V}\subset \R^3$.
We compute 
$$ r_2(U_L)^t\nabla^2 \eta(U_L)f''(U_L) = \begin{pmatrix}
    -2& 0 & 2\alpha \\
    0 & 2 & 0\\
    2\alpha & 0 &-6
\end{pmatrix} $$
which we also note indicates  $r_2$ is correctly oriented, as $\nabla \lambda_2(U_L)r_2(U_L) = l_2(U_L)f''(U_L):r_2(U_L)\otimes r_2(U_L) = 2 > 0$ by equation~\eqref{eq:sj}.
From here we observe, given $v = v_1r_1(U_L) + v_3r_3(U_L)$, 
\begin{align}
    [-C \nabla^2\eta(U_L)(f'(U_L)-\lambda_2(U_L)I)& + r_2(U_L)^t\nabla^2 \eta(U_L)f''(U_L)]:v\otimes v \nonumber\\
    &= \begin{pmatrix}
        -2C-2& 0 & 2\alpha \\
        0 & 2 & 0\\
        2\alpha & 0 &2C -6 
    \end{pmatrix}:v\otimes v \nonumber\\
    &= (-2C-2)v_1^2 + (2C - 6)v_3^2 -4\alpha v_1v_3 \nonumber \\
    &\leq (-2C-2 +2 |\alpha|)v_1^2 + (2C - 6+2|\alpha|)v_3^2 \nonumber
\end{align}
where the last line follows by Young's inequality on $v_1v_2$. 
From this we can ensure strict negative definiteness on the subspace $\text{span}(\{r_1(U_L), r_3(U_L)\})$ when $|\alpha| < 4$ and $C$ satisfies $$|\alpha| - 1 < C < 3 - |\alpha| $$
we see the matrix $[-C \nabla^2\eta(f'-\lambda_2I) + r_2^t\nabla^2 \eta f'']$ is negative definite on $\text{span}(\{r_1,r_3\})$ at $U_L$.

\section{Application to MHD}
Consider the following 2D isentropic MHD system.
$$ \begin{cases}
    v_t - u_x &= 0\\
    (vB)_t - \beta w_x &= 0\\
    u_t + (p + \frac{1}{2}B^2) &= 0 \\
    w_t - \beta B_x &= 0
\end{cases}$$
where $v$ is specific volume and the velocity $(u,w)$ and magnetic field $(\beta, B)$ depend on just one direction. 
This property is achieved when the initial data depends on only one direction. 
We further observe that the divergence free condition of MHD gives us that $\beta$ is constant. 
See \cite{zumbrunMHD} for a study of this system and \cite{kangMHD} for non-contraction of intermediate families against large perturbations for the 3D isentropic MHD system. 
We take our pressure to be $$p(v) = v^{-\gamma}$$ for some $\gamma > 1$. Our entropy is 
$$ \eta(U) = \int_v^\infty p(s)\,ds + \frac{1}{2}(u^2 + w^2) + \frac{vB^2}{2}.$$
In \cite{kang16} it was shown that interior shocks of this system cannot satisfy $a$-{}contraction for generic $\mathcal{S}_{weak}$ perturbations, regardless of shock size. 
Here we will show that for a large number of states $U_L$ small shocks of the interior families with left state $U_L$ cannot be a local attractor with respect to $a$-contraction.

We begin by computing the jacobian of our flux and the corresponding eigenvalues/eigenvectors.
Unlike \cite{kang16} we perform our computations in the conservative variables $v, q:=vB, u, w$. 
We find
$$f'(U) = \begin{pmatrix}
    0 & 0 & -1 & 0 \\
    0 & 0 & 0 & -\beta \\
    p'(v) - \frac{q^2}{v^3} & \frac{q}{v^2} & 0 & 0\\
    \beta \frac{q}{v^2} & -\frac{\beta}{v} & 0 & 0
\end{pmatrix}$$ 
which has characteristic equation
$$ \lambda^4 - \left(\frac{q^2}{v^3} + \frac{\beta^2}{v} +p'(v)\right)\lambda^2 + \frac{\beta^2}{v} p'(v) = 0 $$
giving the four eigenvalues
$$ \lambda_1 = -\sqrt{\alpha_+},\quad \lambda_2 = -\sqrt{\alpha_-},\quad \lambda_3 = \sqrt{\alpha_-},\quad \lambda_4 = \sqrt{\alpha_+}, $$
where $$\alpha_\pm = \frac{1}{2} \left[\frac{q^2}{v^3} + \frac{\beta^2}{v} + c^2 \pm \sqrt{\left(\frac{q^2}{v^3} + \frac{\beta^2}{v} + c^2\right)^2 - 4\beta^2 \frac{c^2}{v}}\right],$$
and $c = \sqrt{-p'(v)}$ is the sound speed. 
The corresponding eigenvectors are 
\begin{align*}
    r_1(U) &= \left(1, \frac{q}{v} - \frac{\alpha_+ - c^2}{q}v^2, \sqrt{\alpha_+}, -\beta v \frac{\alpha_+ - c^2}{q\sqrt{\alpha_+}}\right)^t, \\
    r_2(U) &= \left(1, \frac{q}{v} - \frac{\alpha_- - c^2}{q}v^2, \sqrt{\alpha_-}, -\beta v \frac{\alpha_- - c^2}{q\sqrt{\alpha_-}}\right)^t, \\
    r_3(U) &= \left(-1, -\frac{q}{v} + \frac{\alpha_- - c^2}{q}v^2, \sqrt{\alpha_-}, -\beta v \frac{\alpha_- - c^2}{q\sqrt{\alpha_-}}\right)^t, \\ 
    r_4(U) &= \left(-1, -\frac{q}{v} + \frac{\alpha_+ - c^2}{q}v^2, \sqrt{\alpha_+}, -\beta v \frac{\alpha_+ - c^2}{q\sqrt{\alpha_+}}\right)^t
\end{align*}
These vectors are identically those in \cite{kang16} written in the conservative coordinates, hence they satisfy $\nabla \lambda_i r_i > 0$. 
Henceforth we restrict ourselves to states with $q \ne 0$. 
The hessian of our entropy is 
$$ \nabla^2 \eta(U) = \begin{pmatrix}
    -p'(v) + \frac{q^2}{v^3} & -\frac{q}{v^2} & 0 & 0\\
    -\frac{q}{v^2} & \frac{1}{v} & 0 & 0 \\
    0 & 0 & 1 & 0 \\
    0 & 0 & 0 & 1
\end{pmatrix}.$$

For the intermediate family $i = 2$ we compute 
\begin{equation}
\begin{aligned}r_2(U)^t \nabla^2 \eta(U) f''(U) =& 
\sqrt{\alpha_-}\begin{pmatrix}
    p''(v) + 3\frac{q^2}{v^4} & - 2\frac{q}{v^3} & 0 & 0 \\
    - 2\frac{q}{v^3} & \frac{1}{v^2} & 0 & 0 \\
    0 & 0 & 0 & 0\\
    0 & 0 & 0 & 0
\end{pmatrix}  \\
&+ \beta v \frac{|\alpha_- - c^2|}{q\sqrt{\alpha_-}}\begin{pmatrix}
    -2\beta \frac{q}{v^3} & \beta \frac{1}{v^2} & 0 & 0 \\
    \beta \frac{1}{v^2} & 0 & 0 & 0 \\
    0 & 0 & 0 & 0\\
    0 & 0 & 0 & 0
\end{pmatrix} 
\end{aligned}\label{eq:mhd-crit-matrix}
\end{equation}
Let $V = \text{span}(\{r_1(U),r_3(U),r_4(U)\}).$
We observe the vector $$v_1 = r_1(U) + r_4(U) \in V \cap \ker(r_2(U)^t \nabla^2 \eta(U) f''(U)).$$
We further note that 
$$\nabla^2\eta(U)[f'(U)-\lambda_2(U)I]:v_1\otimes v_1 = -\lambda_2(U) \nabla^2 \eta(U):v_1\otimes v_1 > 0.$$ 
Hence for $-C \nabla^2 \eta(U)[f'(U)-\lambda_2(U)I] + r_2(U)^t\nabla^2 \eta(U) f''(U)$ to be negative semidefinite we must require $C \geq 0$.
Likewise, considering $v_2 = r_1(U) + r_3(U) \in V$ we find 
$$r_2(U)^t \nabla^2 \eta(U) f''(U):v_2\otimes v_2 = v^2\sqrt{\alpha_-} \left(\frac{\alpha_- - \alpha_+}{q}\right)^2 > 0,$$
and 
\begin{equation}\label{eq:mhd-r1r3}
\begin{aligned} 
    \nabla^2 \eta(U)[f'(U)-\lambda_2(U)I]:v_2\otimes v_2 =& (\lambda_1 - \lambda_2 )\nabla^2 \eta(U):r_1(U)\otimes r_1(U)\\
    &+ (\lambda_3- \lambda_2 )\nabla^2 \eta(U):r_3(U)\otimes r_3(U).
\end{aligned}
\end{equation}
We next note that 
\begin{equation}\nabla^2 \eta(U):r_3(U)\otimes r_3(U) \leq \nabla^2 \eta(U):r_1(U)\otimes r_1(U)\label{eq:mhd-eta-diff} \end{equation}
     for all states $U \in \mathcal{V}$ with $v$ sufficiently large (depending on $\beta$.)
This can be shown by computing
\begin{equation} 
\begin{aligned}
    \nabla^2 \eta(U):r_1(U)\otimes r_1(U) - \nabla^2 \eta(U):&r_3(U)\otimes r_3(U) \\
    = \alpha_+ - \alpha_- + \frac{v^2}{q^2}\bigg(&v( (\alpha_+-c^2)^2 - (\alpha_--c^2)^2) \\
    &+\beta^2\left(\frac{(\alpha_+ - c^2)^2}{\alpha_+} - \frac{(\alpha_- - c^2)^2}{\alpha_-}\right)\bigg) \\
    = \alpha_+ - \alpha_- + \frac{v^2}{q^2}\bigg(&v (\alpha_+ - \alpha_-)\left(\frac{q^2}{v^3} + \frac{\beta^2}{v} - c^2\right) \\
    &+\beta^2\left(\frac{(\alpha_+ - \alpha_-)(\beta^2c^2/v - c^4)}{\alpha_+\alpha_-}\right)\bigg) \\
\end{aligned}\end{equation}
From here we recall $\alpha_+ \geq \alpha_-$ and $\beta^2/v > c^2$ for sufficiently large $v$ due to $\gamma > 1$, establishing equation~\eqref{eq:mhd-eta-diff} is positive and we thus satisfy the inequality~\eqref{eq:mhd-eta-diff} for states $U$ with sufficiently large $v$. 
Hence the expression~\eqref{eq:mhd-r1r3} is non-positive for states $U$ satisfying
\begin{equation} (\lambda_1(U) - \lambda_2(U)) + (\lambda_3(U) - \lambda_2(U)) = \lambda_1(U) + 3 \lambda_3(U) \leq 0\label{eq:mhd-assum}\end{equation}
and satisfying the inequality~\eqref{eq:mhd-eta-diff}. 
This condition on the characteristics speeds~\eqref{eq:mhd-assum} is equivalent to our state variables satisfying 
$$ 100\frac{\beta^2 c^2}{v} \leq 9\left(\frac{q^2}{v^3} + \frac{\beta^2}{v} + c^2\right). $$
Again, noting that for any $\gamma > 1$ we have $c^2 \to 0$ as $v \to \infty$, we can conclude this inequality is satisfied for all states $U$ with $v$ sufficiently large.
This establishes for such $U$ there exists no $C \in \R$ such that $[-C \nabla^2 \eta(U)[f'(U)-\lambda_2(U)I] + r_2(U)^t \nabla^2 \eta(U)f''(U)]:v\otimes v$ is non-positive for both $v = v_1,v_2$.
By Corollary \ref{cor:converse} we conclude that no sufficiently small shock with left state $U_L$ is a local attractor for the $a$-contraction theory when the associated specific volume $v_L$ is sufficiently large. 
This is, again, in stark contrast with the extremal case where there exists a weight $a$ giving $a$-contraction against large perturbations in the class $\mathcal{S}_{weak}$ for generic left states $U_L \in \mathcal{V}$.

\printbibliography
\end{document}